\documentclass[11pt,twoside,a4paper]{article}
\usepackage{blindtext}
\usepackage[utf8]{inputenc}
\usepackage[english]{babel}
\usepackage{units}
\usepackage{fullpage}
\usepackage{amssymb,amsmath,amsfonts,amsthm}
\usepackage{geometry}
\usepackage{graphicx}
\usepackage[autostyle]{csquotes}
\usepackage[tight,hang]{subfigure}

\newtheorem{theorem}{Theorem}[section]

\newtheorem{proposition}[theorem]{Proposition}
\newtheorem{corollary}[theorem]{Corollary}
\newtheorem{definition}[theorem]{Definition}
\theoremstyle{remark}

\def\d{\delta}

\newcommand{\hol}{{\mathcal Hol}}
\newcommand{\holo}{{\mathcal Hol}(\Omega)}
\newcommand{\ttn}{\widetilde{T}_N}
\newcommand{\so}{\mathcal{S}(\Omega)}
\newcommand{\ttln}{\widetilde{T}_{\lambda_n}}
\newcommand{\sto}{\mathcal{S}_t(\Omega)}
\newcommand{\uo}{U(\Omega)}
\newcommand{\uoz}                                                                                                    {U(\Omega, \zeta)}

\title{Boundedness of derivatives and anti-derivatives of holomorphic functions as a rare phenomenon}
\author{Maria Siskaki}
\date{}

\begin{document}
\maketitle

\begin{abstract}
  In this article we prove a general result which in particular suggests that, on a simply connected domain $\Omega$ in $\mathbb{C}$, all the derivatives and anti-derivatives of the generic holomorphic function are unbounded. A similar result holds for the operator $\ttn$ of partial sums of the Taylor expansion with center $\zeta \in \Omega$ at $z=0$, seen as functions of the center $\zeta$. We also discuss a universality result of these operators $\ttn$.
\end{abstract}
\vspace{5pt}
\begin{flushleft} AMS Classification n$^o$: primary 30K99, secondary 30K05
  \end{flushleft}
\vspace{5pt}
\begin{flushleft}
  Key words and phrases: Baire's Theorem, generic property, Differentiation and Integration operators, Taylor expansion, partial sums, universal Taylor series
\end{flushleft}

\section{Introduction}
\par
Let $\Omega$ be a domain in the complex plane and consider the space $\holo$ of all the functions that are holomorphic on $\Omega$ with the topology of uniform convergence on compacta. In the first section of this article we show that, for a function $f \in \holo$, the phenomenon of its $k$- derivative or $k$-anti-derivative being bounded on $\Omega$ is a rare phenomenon in the topological sense, provided that $\Omega$ is simply connected. We do this by using Baire's Theorem and we prove that the set $\mathcal{D}$ of all the functions $f \in \holo$ with the property that the derivatives and the anti-derivatives of $f$ of all orders are unbounded on $\Omega$ is a dense $G_\d$ set in $\holo$.
\par
If a function $f$ is holomorphic in an open set containing $\zeta$, then $S_N(f,\zeta)(z)$ denotes the $N$-th partial sum of the Taylor expansion of $f$ with center $\zeta$ at $z$.
If $\Omega$ is a simply connected domain and $\zeta \in \Omega$, we define the class $\uoz$ as follows:
\begin{definition}
  The set $\uoz$ is the set of all functions $f \in \holo$ with the property that, for every compact set $K \subset \mathbb{C}$, $K \cap \Omega = \emptyset$, with $K^{\mathsf{c}}$ connected, and for every function $h$ which is continuous on $K$ and holomorphic in the interior of $K$, there exists a sequence $\{\lambda_n\} \in \{0,1,2,...\}$ such that
  \begin{equation*}
  \sup\limits_{z \in K}|S_{\lambda_n}(f,\zeta)(z)-h(z)| \longrightarrow 0, \hspace{10pt} n \rightarrow \infty
\end{equation*}
\end{definition}

Denote $\mathbb{D}= \{z \in \mathbb{C}: |z|<1 \}$. It is shown in \cite{nestoridis1996universal} that $U(\mathbb{D},0)$ is a  dense $G_\d$ set in $\hol (\mathbb{D})$. More generally, in \cite{nestoridis1999extension} it is shown that $\uoz$ is a dense $G_\d$ set in $\holo$, where $\Omega$ is any simply connected domain and $\zeta \in \Omega$. Next, for $\Omega$ as above, we define the set $\uo$:
\begin{definition}
  The set $\uo$ is the set of all functions $f \in \holo$ with the property that,  for every compact set $K \subset \mathbb{C}$, $K \cap \Omega = \emptyset$, with $K^{\mathsf{c}}$ connected, and every function $h$ which is continuous on $K$ and holomorphic in the interior of $K$, there exists a sequence $\{\lambda_n\} \in \{0,1,2,...\}$ such that, for every compact set $L \subset \Omega$,
\begin{equation*}
  \sup\limits_{\zeta \in L}\sup\limits_{z \in K}|S_{\lambda_n}(f,\zeta)(z)-h(z)| \longrightarrow 0, \hspace{10pt} n \rightarrow \infty
\end{equation*}
\end{definition}

Again in \cite{nestoridis1999extension} it is shown that $\uo$ is a dense $G_\d$ set in $\holo$. Furthermore, in \cite{melas2001universality} it is shown that $\uoz = \uo$, provided that $\Omega$ is contained in a half-plane. This result is generalized in \cite{muller2006universal}, where it is shown that $\uoz = \uo$ for any simply connected domain $\Omega$ and $\zeta \in \Omega$.
\par
In the second section of this article, we fix a $\zeta_0 \in \Omega$ and, for $N \geq 1$, we consider the function
\begin{align*}
  S_N(f,\zeta_0): \mathbb{C} & \rightarrow \mathbb{C} \\
  z& \mapsto \sum_{n=0}^{N}\frac{f^{(n)}(\zeta_0)}{n!}(z-\zeta_0)^n = S_N(f,\zeta_0)(z)
\end{align*}

 \par
V. Nestoridis suggested that, contrary to the functions in $ \uoz$, whose Taylor partial sums are considered as functions of $z$ with the center $\zeta$ fixed, we fix $z=0$ and let the center $\zeta$ vary in $\Omega$. Thus, for $N \geq 0$, we obtain an operator
\begin{align*}
  \ttn : \holo & \rightarrow \holo \\
  f & \mapsto \ttn (f)
\end{align*}
where
\begin{align*}
  \ttn(f): \Omega & \rightarrow \mathbb{C}\\
  \zeta & \mapsto \sum_{n=0}^{N}\frac{f^{(n)}(\zeta)}{n!}(- \zeta)^n = \ttn (f)(\zeta)
\end{align*}
for any $f \in \holo$ and $N \geq 0$.
 The set of functions $f \in \holo$ such that $\ttn(f)$ is unbounded on $\Omega$ for all $N \geq 0$ is residual in $\holo$. This led V.Nestoridis to conjecture that,  if $0 \notin \Omega$, then the class $\so$ of all functions $f \in \holo$ with the property that the set $\{\ttn (f): N = 0,1,2,...\}$ is dense in $\holo$ is a dense $G_\d$ set in $\holo$. In this article we show that either $\so = \emptyset$ or $\so$ is a dense $G_\d$ set in $\holo$. The question of whether $\so \neq \emptyset$ will be examined in a future article. However, we do show that, if $0 \notin \Omega$, then the set $\sto$ of the functions $f \in \holo$ with he property that the closure of the set $\{\ttn (f)\}$ contains the constant functions on $\Omega$ is residual in $\holo$. We do this by proving that $\sto$ contains the set $\uo$, which is already proven to be a dense $G_\d$ set in $\holo$ (\cite{nestoridis1999extension}).
\par
In the last part of the article, answering a question by T. Hatziafratis, we prove that, for a countable set $E \subset \mathbb{T}=\{z \in \mathbb{C}: |z|=1\}$, the generic holomorphic function on $\mathbb{D}$ has unbounded derivatives and anti-derivatives on each ray $[0,z)$, $z \in E$. We also obtain a more general result, where in fact we do not use Baire's Theorem and , therefore, the topological vector space used need not be a Fr\'echet space.

\section{Preliminaries}
\par
Regarding the terminology used, a set $\Omega \subset \mathbb{C}$ is called a \textit{domain} if it is open and connected in $\mathbb{C}$. A $G_\d$ set in $\holo$ is a countable intersection of open sets in $\holo$ and an $F_\sigma$ set is a countable union of closed sets in $\holo$. Furthermore, a subset $E$ of $\holo$ is called \textit{dense} if
there exists no non-empty open subset $U$ of $\holo$ such $U$ and $E$ are
disjoint. The set $E$ is \textit{nowhere dense} in $\holo$ if every
non-empty open set $U$ has an open non-empty subset $V$ such that
$E$ and $V$ are disjoint. This is equivalent to the closure of $E$ having an empty interior in $\holo$. A \textit{set of the first category} in $\holo$ is a set that can be expressed as a countable union of nowhere dense sets in $\holo$. A $G_\d$ dense subset of $\holo$ is a $G_\d$ subset which is also dense. Because the space $\holo$ is metrizable complete, Baire's theorem implies that a subset of $\holo$ is $G_\d$ dense iff it is the countable intersection of open and dense subsets of $\holo$. A subset of $\holo$ is called residual if it contains a $G_\d$ dense set. Equivalently, if its complement is contained in an $F_\sigma$ set of the first category.   \\
\par
Let $\Omega_1, \Omega_2$ be two domains in $\mathbb{C}$ and $T: \hol(\Omega_1) \rightarrow \hol(\Omega_2)$ be a linear operator with the property that for every $z \in \Omega_2$, the function $ f \mapsto T(f)(z)$ is continuous in $\hol(\Omega_1)$. Observe that this latter property is weaker than $T$ being continuous. Define
\begin{equation*}
  \mathcal{U}_T= \big\{f \in \hol(\Omega_1): T(f) \hspace{3pt} \text{  is unbounded on } \Omega_2 \big\}
\end{equation*}
\begin{proposition}
\label{eitheror}
  If $\Omega_1, \Omega_2 $ are two domains in $\mathbb{C }$ and $T$ is as above, then either $\mathcal{U}_T = \emptyset$ or $\mathcal{U}_T$ is a dense $G_\delta$ set in $\hol(\Omega_1)$.
\end{proposition}
\begin{proof}
If $\mathcal{U}_T \neq \emptyset$, for $m \geq 1$ define
\begin{equation*}
  U_m= \big\{f \in \hol(\Omega_1): |T(f)(z)| \leq m \hspace{5pt} \text{ for all } z \in \Omega_2 \big\}
\end{equation*}
  Then
  \begin{equation*}
  \mathcal{U}_T= \Big(\bigcup\limits_{m=1}^{\infty}U_m \Big)^{\mathsf{c}}= \bigcap\limits_{m=1}^{\infty}U_m^{\mathsf{c}}
  \end{equation*}
  We will show that $U_m$ is closed and nowhere dense in $\hol(\Omega_1)$ for each $m \geq 1$.
  \par
  To see that it is closed, take a sequence $\{f_n\}$ in $U_m$ such that $f_n \longrightarrow f$ uniformly on compact subsets of $\Omega_1$ for some function $f$. Then $f \in \hol(\Omega_1)$ and, for $z \in \Omega_2$ we have
  \begin{align*}
    |T(f)(z)| & \leq |T(f)(z)-T(f_n)(z)| + |T(f_n)(z)|\\
    & \leq |T(f-f_n)(z)| +m
  \end{align*}
  Taking $n \rightarrow \infty$ we get that $|T(f)(z)| \leq m$ because of the continuity of $ f \mapsto T(f)(z)$, i.e. $f \in U_m$. Thus, $U_m$ is closed.\par
  To see that $U_m$ is nowhere dense, it suffices to show that $U_m^{\circ}= \emptyset$. Suppose $f \in U_m^{\circ} $. Since $\mathcal{U}_T \neq \emptyset$, there exists a function $g \in \hol(\Omega_1)$ such that $T(g)$ is unbounded on $\Omega_2$. Then $\{f+\tfrac{1}{n}g\}_n$ is a sequence in $\hol(\Omega_1)$ and, if $K$ is a compact subset of $\Omega_1$, we have
  \begin{align*}
    \| (f+\tfrac{1}{n}g)-f\|_K &= \sup\limits_{z \in K} |f(z)+\frac{1}{n} g(z)-f(z)| \\ &= \sup\limits_{z \in K}|\frac{1}{n}g(z)|=\frac{1}{n} \|g\|_K
  \end{align*}
  By taking $n \rightarrow \infty$ and observing that $\|g\|_K < \infty$,  $g$ being holomorphic on $\Omega_1 \supset K$, we obtain that  $f+\frac{1}{n}g \longrightarrow f$ uniformly on $K$. But $K$ was an arbitrary compact subset of $\Omega_1$, so   $f+\tfrac{1}{n} \hspace{2pt}g \longrightarrow f$ uniformly on compact subsets of $\Omega_1$. \\
  Since $f \in U_m^{\circ}$, there exists an $n_0$ such that $f+\tfrac{1}{n_0} \hspace{2pt}g \in U_m$. By the linearity of $f \mapsto T(f)$ this means that
  \begin{align*}
    \tfrac{1}{n_0} \hspace{2pt}|T(g)(z)| & \leq |T(f)(z)+\tfrac{1}{n_0} \hspace{2pt} T(g)(z)|+|T(f)(z)| \\
    & \leq m+m
  \end{align*}
  or $|T(g)(z)| \leq 2mn_0$, for all $z \in \Omega_2$, which is contradictory to the fact that $T(g)$ is unbounded on $\Omega_2$. Thus, $U_m^{\circ}= \emptyset$ and the proof is complete.
\end{proof}

\begin{proposition}
\label{capuT_n}
  For $n \in \mathbb{Z}$, let $T_n: \hol(\Omega_1) \rightarrow \hol(\Omega_2)$ be linear and such that for every $z \in \Omega_2$, the function $ f \mapsto T(f)(z)$ is continuous in $\hol(\Omega_1)$. If $\mathcal{U}_{T_n} \neq \emptyset$ for all $n \in \mathbb{Z}$ then the set $\bigcap\limits_{n} \mathcal{U}_{T_n}$ is dense $G_\delta $ in $\hol(\Omega_1)$.
\end{proposition}
\begin{proof}
The space $\hol(\Omega_1)$ with the metric of uniform convergence on compacta is a complete metric space, so by Baire's Theorem any countable intersection of dense $G_\d$ sets in $\hol(\Omega_1)$ is again a dense $G_\d$ set in $\hol(\Omega_1)$. Since $\mathcal{U}_{T_n} \neq \emptyset$, it is a dense $G_\d$ set in $\holo$ by Proposition (\ref{eitheror}), $n \in \mathbb{Z}$, and the desired result follows immediately.
\end{proof}
Observe that Propositions (\ref{eitheror}) and (\ref{capuT_n}) still hold if we replace $\hol(\Omega_2)$ by $\mathbb{C}^{X}$, where $X$ is any non-empty set and $\mathbb{C}^{X}$ is the set of all functions from $ X $ to $\mathbb{C}$.

\section{Boundedness of derivatives and anti-derivatives as a rare phenomenon}
\begin{proposition}
\label{a_0}
  Let $\Omega \subset \mathbb{C}$ be open and non-empty. The set $\mathcal{A}_0$ of all functions $f \in \hol(\Omega)$ that are bounded on $\Omega$ is a set of the first category in $\hol(\Omega)$.
\end{proposition}
\begin{proof}
  For $m \in \mathbb{N}$ define
  \begin{equation*}
   A_m = \Big\{ f \in \hol(\Omega): |f(z)| \leq m, \text{ for all } z \in \Omega \Big\}
 \end{equation*}

   It is obvious that
   \begin{equation*}
   \mathcal{A}_0 = \bigcup\limits_{m=1}^{+\infty} A_m
   \end{equation*}
    We will show that every $A_m$ is closed and has an empty interior in $\holo$.
    \par
   For $m \in \mathbb{N}$, the set $A_m$ is closed in $\hol(\Omega)$: Let $\{f_n\}$ be a sequence in $A_m$ and $f$ a function on $\Omega$ such that $f_n \longrightarrow f$ uniformly on compact subsets of $\Omega$. By the Weierstrass theorem, $f\in \hol(\Omega)$ and, for $z \in \Omega$
   \begin{equation*}
     |f(z)|= \lim\limits_{n \longrightarrow \infty} |f_n(z)| \leq m
   \end{equation*}
  Therefore, $f \in A_m$ and $A_m$ is closed in $\hol(\Omega)$ for each $m=1,2,...$.
  \par
  Next we show that  $A_m ^{\circ} = \emptyset $ for all $m=1,2,...$:
  First observe that there exists a function $g \in \hol(\Omega)$ that is unbounded on $\Omega$. Indeed, if $\Omega$ is unbounded take $g(z)=z,$ $z \in \Omega$, and if $\Omega$ is bounded, take $\zeta_0 \in \partial \Omega$ and $g(z)= \frac{1}{z-\zeta_0}$.\\
  Now assume that there exists $f \in A_m^{\circ}$ for some fixed $m=1,2,...$. Then $\{f+\tfrac{1}{n}g\}_n$ is a sequence in $\hol(\Omega)$ and  $f+\frac{1}{n}g \longrightarrow f$ uniformly on compact subsets of $\Omega$,  $n\rightarrow \infty$. But $f \in A_m^{\circ}$, hence there exists an $n_0 \in \mathbb{N}$ such that $f+ \tfrac{1}{n_0}g \in A_m^{\circ}$. This means that
  \begin{equation*}
    |f(z)+\frac{1}{n_0}g(z)| \leq m, \text{ for all } z \in \Omega
  \end{equation*}
  But then, for any $z \in \Omega$ we would have
  \begin{align*}
    |\frac{1}{n_0}g(z)| &= |f(z)+\frac{1}{n_0}g(z)-f(z)| \\
    & \leq |f(z)+\frac{1}{n_0}g(z)| +|f(z)| \\
    & \leq m+m,
  \end{align*}
  Therefore, $|g(z)| \leq 2mn_0$ for all $z \in \Omega$, which is contradictory to the fact that $g$ is unbounded on $\Omega$. Thus, $A_m^{\circ}= \emptyset $ and the proof is complete.
  \end{proof}
  For $f \in \holo$, we denote by $f^{(k)}$ the $k$-derivative of $f$, $k \geq 1$. By $f^{(0)}$ we denote $f$ itself.
  \begin{proposition}
\label{a_k}
  Let $\Omega \subset \mathbb{C}$ be open and non-empty and $k \in \mathbb{N}$. The set $\mathcal{A}_k$ of all functions $f \in \hol(\Omega)$ such that $f^{(k)}$ is bounded on $\Omega$ is a set of the first category in $\hol(\Omega)$.
\end{proposition}
\begin{proof}
 For $m \in \mathbb{N}$, define
 \begin{equation*}
   A_m = \Big\{ f \in \hol(\Omega): |f^{(k)}(z)| \leq m, \text{ for all } z \in \Omega \Big\}
 \end{equation*}

 It is obvious that
 \begin{equation*}
 \mathcal{A}_k = \bigcup\limits_{m=1}^{+\infty} A_m
 \end{equation*}
  We will show that each $A_m$ is closed and has empty interior in $\hol(\Omega)$.
   \par
  To see that it is closed, take a sequence $\{f_n\}$ in $A_m$ and a function $f$ on $\Omega$ such that $f_n \longrightarrow f$ uniformly on compact subsets of $\Omega$. By the Weierstrass theorem we have that $f \in \hol(\Omega)$ and $f^{(k)}_n \longrightarrow f^{(k)}$ uniformly on compact subsets of $\Omega$. Therefore, for any $z \in \Omega$ we have that
  \begin{equation*}
     |f^{(k)}(z)|= \lim\limits_{n \rightarrow \infty} |f^{(k)}_n(z)| \leq m
   \end{equation*}
   i.e. $f \in A_m$. Thus, $A_m$ is closed.
   \par
   To see that $A^{\circ}_m= \emptyset$, first observe that there exists a function $g \in \hol(\Omega)$ such that $g^{(k)}$ is unbounded on $\Omega$. Indeed, if $\Omega$ is unbounded take $g(z)=z^{k+1}$ and if $\Omega$ is bounded take $\zeta_0 \in \partial \Omega$ and $g(z)= \frac{1}{z-\zeta_0}$.\\
  Now assume that there exists $f \in A_m^{\circ}$. Then $\{f+\tfrac{1}{n}g\}_n$ is a sequence in $\hol(\Omega)$ and $f+\frac{1}{n}g \longrightarrow f$ uniformly on compact subsets of $\Omega$,  $n\rightarrow \infty$. But $f \in A_m^{\circ}$, hence there exists an $n_0 \in \mathbb{N}$ such that $f+ \tfrac{1}{n_0}g \in A_m^{\circ}$. This means that
  \begin{equation*}
    |f^{(k)}(z)+\frac{1}{n_0}g^{(k)}(z)| \leq m, \text{ for all } z \in \Omega
  \end{equation*}
where the linearity of the derivative operator is used.
  But then, for any $z \in \Omega$ we would have
  \begin{align*}
    |\frac{1}{n_0}g^{(k)}(z)| &= |f^{(k)}(z)+\frac{1}{n_0}g^{(k)}(z)-f^{(k)}(z)| \\
    & \leq |f^{(k)}(z)+\frac{1}{n_0}g^{(k)}(z)| +|f^{(k)}(z)| \\
    & \leq m+m,
  \end{align*}
  Thus $|g^{(k)}(z)| \leq 2mn_0$ for all $z \in \Omega$, which is contradictory to the fact that $g^{(k)}$ is unbounded on $\Omega$. Thus, $A_m^{\circ} = \emptyset$ and the proof is complete.
\end{proof}

  \begin{proposition}
  \label{e}
  Let $\Omega \subset \mathbb{C}$ be open and non-empty. The set $\mathcal{E}$ of all functions $f \in \hol(\Omega)$ with the property that $f^{(k)}$ is unbounded on $\Omega$, for all $k \in \mathbb{N}$, is a dense $G_{\delta}$ set in $\hol(\Omega)$.
\end{proposition}

\begin{proof}
  Using the notation previously established it is obvious that
  \begin{equation*}
    \mathcal{E}= \bigcap\limits_{k=0}^{\infty} \mathcal{A}_k^\mathrm{c}
  \end{equation*}
  By Propositions (\ref{a_0}) and (\ref{a_k}) we have that for each $k \geq 0$ , the set $\mathcal{A}_k$ is the countable union of closed, nowhere dense sets in $\hol(\Omega)$, so its complement $\mathcal{A}_k^{\mathsf{c}}$ must be a dense $G_{\delta}$ set in $\hol(\Omega)$. By Baire's Theorem, the set $\mathcal{E}$ is a dense $G_{\delta}$ set in $\hol(\Omega)$ as a countable intersection of dense $G_\d$ sets in a complete metric space.
\end{proof}

From now on, and throughout the remainder of this section, consider an $\Omega \subset \mathbb{C}$ which is non-empty, open and simply connected. Fix $\zeta_0 \in \Omega$ and, for $f \in \hol(\Omega)$ define
\begin{align*}
T(f)(z)&= \int_{\gamma_z}f(\xi) d\xi, \hspace{90pt} \text{ for all } z \in \Omega  \\
T^{(k)}(f)(z)&= \int_{\gamma_z}T^{(k-1)}(f)(\xi) d\xi, \hspace{50pt} \text{ for all } z \in \Omega, k \geq 2
\end{align*}
where $\gamma_{z}$ is any polygonal line in $\Omega $ that starts at $\zeta_0$ and ends at $z$. Since $\Omega $ is assumed to be simply connected, each $T^{(k)}$ is well-defined and holomorphic in $\Omega$ and its $k-$derivative is $f$.
\begin{proposition}
\label{integraliscontinuous}
  The operator
  \begin{align*}
  T: \hol(\Omega)&\longrightarrow \hol(\Omega)\\
  f& \mapsto T(f)
  \end{align*}
  is linear and continuous on $\hol(\Omega)$.
\end{proposition}
\begin{proof}
  The linearity of $T$ is obvious from the linearity of the integral. For the continuity, take a sequence $\{f_n\}$ in $\hol(\Omega)$ and a function $f$ on $\Omega$ such that $f_n \longrightarrow f$ uniformly on compact subsets of $\Omega$. By the Weierstrass theorem we have that $f \in \hol(\Omega)$. We must show that $T(f_n)\longrightarrow T(f)$ on compact subsets of $\Omega$. \par
  Let $K$ be a compact subset of $\Omega$. Either $\Omega =\mathbb{C}$ or $\Omega \neq \mathbb{C}$.\par
   In the first case, i.e. $\Omega = \mathbb{C}$, for $z \in K$ we take $\gamma_z$ to be the line segment $[\zeta_0,z]$. Set $M=\max\{|\zeta_0|, \max\limits_{z \in K}|z|\}$ and observe that $M$ is well defined and finite because $K$ is compact in $\mathbb{C}$. Define $L=\overline{D(0,M)}= \{z \in \mathbb{C}: |z| \leq M\} $. Then $L$ is compact in $\mathbb{C}$, $K \subset L$ and $\gamma_z \subset L$, for all $z \in K$. Therefore, for $z\in K$ we have
  \begin{align*}
    |T(f_n)(z)-T(f)(z)| &= \big| \int_{\gamma_z}f_n(\xi)d\xi-\int_{\gamma_z}f(\xi)d\xi \hspace{3pt}\big| \\
   & = \big|\int_{\gamma_z}(f_n(\xi)-f(\xi))d\xi \hspace{3pt}\big|\\
    & \leq \|f_n-f\|_L \hspace{5pt} |z-\zeta_0|\\
    & \leq 2 M \|f_n-f\|_L
  \end{align*}
 Thus $\|T(f_n)-T(f)\|_K \leq 2M \|f_n-f\|_L \longrightarrow 0 $, $n\rightarrow \infty$. \par
 In the second case, i.e. $\Omega \neq \mathbb{C}$, since $\Omega $ is a simply connected domain, by the Riemann Mapping Theorem there exists an analytic function $\phi :\mathbb{D} =\{z \in \mathbb{C}: |z|<1\} \longrightarrow \mathbb{C}$ such that $\phi$ is univalent and $\phi(\mathbb{D})= \Omega$. Obviously $\phi$ is a homeomorphism between $\mathbb{D}$ and $\Omega$. Since the set $\{\zeta_0\}\cup K \subset \Omega$ is compact, the set $\phi ^{-1}(\{\zeta_0\}\cup K)\subset \mathbb{D}$ is also compact. Therefore, there exists an $r$, with $0<r<1$, such that $\phi ^{-1}(\{\zeta_0\}\cup K)\subset \overline{D(0,r)}=\{z \in \mathbb{C}: |z| \leq r\}$. Define $L=\phi(\overline{D(0,r)})\subset \phi(\mathbb{D})=\Omega$. Then $L$ is compact and $K\subset L$. For $z \in K$ we have that $\phi^{-1}(\zeta_o)$, $\phi^{-1}(z) \in \overline{D(0,r)}$, hence the line segment $[\phi^{-1}(\zeta_o),\phi^{-1}(z)] \subset \overline{D(0,r)}$. Therefore, if $\sigma: [0,1]\longrightarrow \mathbb{C}$ is a parametrization of $[\phi^{-1}(\zeta_o),\phi^{-1}(z)]$, then $Length(\sigma) \leq 2r$.
 Take $\gamma_z= \phi([\phi^{-1}(\zeta_o),\phi^{-1}(z)])\subset \phi(\overline{D(0,r)})=L$ and observe that $\gamma_z$ is rectifiable: $\phi \circ \sigma :[0,1]\longrightarrow \Omega $ is a parametrization of $\gamma_z$ and
 \begin{align*}
 Length(\gamma_z)&= \int_{0}^{1} |\gamma_{z}^{'}(t)|dt \\ &= \int_{0}^{1}|(\phi \circ \sigma)^{'}(t)|dt \\
 &= \int_{0}^{1}|(\phi ^{'}(\sigma(t))|\hspace{2pt}|\sigma ^{'}(t)|dt  \\
 & \leq \max \big\{|\phi ^{'}(z)|: z \in \overline{D(0,r)} \big\}  \hspace{5pt}
  Length(\sigma) \\
  & \leq \max \big\{|\phi ^{'}(z)|: z \in \overline{D(0,r)} \big\}  \hspace{5pt} 2r
 \end{align*}
 which is of course finite because $\phi ^{'}$ is continuous on the compact set $\overline{D(0,r)}$.
 \par
 We then have
 \begin{align*}
   |T(f_n)(z)-T(f)(z)| &= \big| \int_{\gamma_z}f_n(\xi)d\xi-\int_{\gamma_z}f(\xi)d\xi \hspace{3pt}\big| \\
   & = \big|\int_{\gamma_z}(f_n(\xi)-f(\xi))d\xi \hspace{3pt}\big|\\
    & \leq \|f_n-f\|_L \hspace{5pt} Length(\gamma_z)\\
    & \leq \|f_n-f\|_L \hspace{5pt} \max \big\{|\phi ^{'}(z)|: z \in \overline{D(0,r)} \big\}  \hspace{5pt} 2r
 \end{align*}

  Thus $\|T(f_n)-T(f)\|_K \leq \|f_n-f\|_L \hspace{5pt} \max \big\{|\phi ^{'}(z)|: z \in \overline{D(0,1)} \big\}  \hspace{5pt} 2r \longrightarrow 0 $, $n\rightarrow \infty$.\\
  In any case we have shown that $T(f_n)\longrightarrow T(f)$ uniformly on $K$. Since $K$ was an arbitrary compact subset of $\Omega$, the continuity of $T$ follows.
  \end{proof}
\begin{corollary}
\label{kprimitiveiscont}
  Let $k \geq 1$. The operator
  \begin{align*}
  T^{(k)}: \hol(\Omega)&\longrightarrow \hol(\Omega)\\
  f& \mapsto T^{(k)}(f)
  \end{align*}
  is linear and continuous on $\hol(\Omega)$.
\end{corollary}
\begin{proof}
We have that $T^{(k)}= T \circ T \circ ... \circ T$, the composition of $T$ $k$ times. Therefore linearity and continuity both follow by Proposition (\ref{integraliscontinuous}).
\end{proof}

\begin{corollary}
\label{pointwisekprimitive}
  If $f_n \longrightarrow f$ uniformly on compact subsets of $\Omega$ and $k \geq 1$, then $T^{(k)}(f_n) \longrightarrow T^{(k)}(f)$ pointwise in $\Omega$.
\end{corollary}
\begin{proof}
  By the Weierstrass Theorem, $f \in \hol(\Omega)$. By Corollary (\ref{kprimitiveiscont}) we have that $T^{(k)}(f_n) \longrightarrow T^{(k)}(f)$ uniformly on compact subsets of $\Omega$ and therefore $T^{(k)}(f_n) \longrightarrow T^{(k)}(f)$ pointwise in $\Omega$.
\end{proof}
\begin{proposition}
\label{b_k}
  Let $\Omega \subset \mathbb{C}$ be a simply connected domain and $k \geq 1$. The set $\mathcal{B}_k$ of all $f \in \hol(\Omega)$ such that $T^{(k)}(f)$ is bounded on $\Omega$ is a set of the first category in $\hol(\Omega)$.
\end{proposition}
\begin{proof}
  For $m \in \mathbb{N}$, define
  \begin{equation*}
    B_m= \big\{f \in \hol(\Omega): |T^{(k)}(f)(z)| \leq m \hspace{5pt} \text{ for all } z \in \Omega \big\}
  \end{equation*}
  Then $\mathcal{B}_k= \bigcup\limits_{m=1}^{\infty}B_m$. We will show that each $B_m$ is closed and nowhere dense in $\hol(\Omega)$. \par
  To see that it is closed, take a sequence $\{f_n\}$ in $B_m$ such that $f_n \longrightarrow f$ uniformly on compact subsets of $\Omega$. By Corollary (\ref{pointwisekprimitive}), $T^{(k)}(f_n) \longrightarrow T^{(k)}(f)$ pointwise in $\Omega$. Therefore, for $z \in \Omega$ we have that
  \begin{align*}
    |T^{(k)}(f)(z)| & \leq |T^{(k)}(f)(z)-T^{(k)}(f_n)(z)| + |T^{(k)}(f_n)(z)|\\
    & \leq |T^{(k)}(f)(z)-T^{(k)}(f_n)(z)| +m
  \end{align*}
  Taking $n \rightarrow \infty$ we obtain $|T^{(k)}(f)(z)| \leq m$ and therefore $f \in B_m$. Thus, $B_m$ is closed.
  \par
  To see that $B_m ^{\circ}= \emptyset$, first observe that there exists a function $g \in \hol(\Omega)$ such that $T^{(k)}(g)$ is unbounded on $\Omega$: indeed, if $\Omega $ is unbounded take $g(z)=1$, $z\in \Omega$, and if $\Omega $ is bounded take $\zeta_0 \in \partial \Omega$ and $g(z)=\frac{1}{(z-\zeta_0)^{k+1}}$. Now assume that $f \in B_m^{\circ}$. Then $f+\tfrac{1}{n} \hspace{2pt}g \longrightarrow f$ uniformly on compact subsets of $\Omega$, $n \rightarrow \infty$. Therefore, there exists an $n_0$ such that $f+\tfrac{1}{n_0} \hspace{2pt}g \in B_m$. By the linearity of $f \mapsto T^{(k)}(f)$ this means that
  \begin{equation*}
    |T^{(k)}(f)(z)+\tfrac{1}{n_0} \hspace{2pt} T^{(k)}(g)(z)|=|T^{(k)}(f+\tfrac{1}{n_0} \hspace{2pt}g)(z)|\leq m
  \end{equation*}
  for all $z \in \Omega$. But then
  \begin{align*}
    \tfrac{1}{n_0} \hspace{2pt}|T^{(k)}(g)(z)| & \leq |T^{(k)}(f)(z)+\tfrac{1}{n_0} \hspace{2pt} T^{(k)}(g)(z)|+|T^{(k)}(f)(z)| \\
    & \leq m+m
  \end{align*}
  or $|T^{(k)}(g)(z)| \leq 2mn_0$, for all $z \in \Omega$, which is contradictory to the fact that $T^{(k)}(g)$ is unbounded on $\Omega$. Thus, $B_m^{\circ} = \emptyset $ and the proof is complete.
\end{proof}
For $f \in \holo$, where $\Omega \subset \mathbb{C}$ is a simply connected domain, we denote
\[
f^{(k)} =
\begin{cases}
\text{the } k^{th} \text{ derivative of }f, & \text{if } k>0\\
f, & \text{if } k=0\\
T^{(-k)}(f),& \text{if } k<0
\end{cases}
\]
where $T^{(k)}(f)$ as defined above. Collecting all the above results together we get
\begin{theorem}
\label{d}
Let $\Omega \subset \mathbb{C}$ be a simply connected domain. Then
  the set $\mathcal{D}$ of all functions $f \in \hol(\Omega)$ with the property that $f^{(k)}$ is unbounded on $\Omega$ for all $k \in \mathbb{Z}$ is a dense $G_\delta$ subset of $\hol(\Omega)$.
\end{theorem}
\begin{proof}
  For $k \in \mathbb{Z}$ define
  \begin{equation*}
    D_k = \big\{ f \in \hol(\Omega): f^{(k)}  \hspace{3pt} \text{ unbounded on } \Omega \big\}
  \end{equation*}
  Then $\mathcal{D} = \bigcap \limits_{k \in \mathbb{Z}}D_k$. By Propositions (\ref{a_0}), (\ref{a_k}) and (\ref{b_k}) we have that each $D_k$ is a dense $G_\delta$ set in $\hol(\Omega)$, because its complement is a countable union of closed, nowhere dense sets in $\hol(\Omega)$. Since $\hol(\Omega)$ is a complete metric space, Baire's Theorem gives that any countable intersection of dense $G_\delta $ sets is again a dense $G_\delta $ set.
\end{proof}
At this point observe that Proposition (\ref{e}) and Theorem (\ref{d}) are immediate corollaries to Proposition (\ref{capuT_n}):\\
 The operator
 \begin{align*}
   \Lambda: \holo & \rightarrow \holo \\
   f& \mapsto f^{'}
 \end{align*}

  is linear and continuous by the Weierstrass Theorem.\\
   If additionally $\Omega$ is simply connected, the same holds for the operator
  \begin{align*}
   \widetilde{ \Lambda}: \holo & \rightarrow \holo \\
   f& \mapsto \int_{\gamma_z}f(\xi)d\xi
  \end{align*}
    by Proposition (\ref{integraliscontinuous}), the primitive of $f$ being defined as in the discussion preceding that same  Proposition.
    \par
     Now define $\Lambda _k$ to be $k$ compositions of $\Lambda $ with itself, $k \geq 1$, $\Lambda_0$ to be the identity function on $\holo$ and $\Lambda_k$ to be $(-k)$ compositions of $\widetilde{\Lambda}$ with itself, $k \leq -1$. Then each $\Lambda_k$ is linear and continuous in $\holo$ and, furthermore, $\mathcal{U}_{\Lambda_k} \neq \emptyset$, for all $k \in \mathbb{Z}$. Therefore, the set $\bigcap\limits_{k \in \mathbb{Z}}\mathcal{U}_{\Lambda_k}$ is a dense $G_\delta $ subset of $\holo$. But this is exactly the set $\mathcal{D}$ of Theorem (\ref{d}).\\
\section{Universality of operators related to the partial sums}

Now assume that $\Omega$ is a domain in $\mathbb{C}$. For $N \geq 0$ we define:
\begin{align*}
 S_N: \holo & \rightarrow \hol(\Omega \times \mathbb{C}) \\
 f & \mapsto S_N(f, \cdot)(\cdot)= S_N(f)
\end{align*}
where
\begin{equation*}
  S_N(f,\zeta)(z)= \sum_{n=0}^{N}\frac{f^{(n)}(\zeta)}{n!}(z- \zeta)^n, \hspace{5pt} \zeta \in \Omega, z\in \mathbb{C}
\end{equation*}
Then $S_N$ is obviously linear. By the Weierstrass Theorem it is also continuous; indeed suppose $K=K_1 \times K_2$ is a compact subset of $\Omega \times \mathbb{C}$, where $K_1, K_2$ are compact subsets of $\Omega$ and $\mathbb{C}$ respectively, and $f_k \longrightarrow f$ uniformly on compact subsets of $\Omega$. Set $M = \max\limits_{(\zeta,z)\in K}|z- \zeta|$. Then, for $(\zeta,z) \in K$ we have that
\begin{align*}
  | S_N(f_k,\zeta)(z)-S_N(f,\zeta)(z) |&= \Big|\sum_{n=0}^{N}\frac{f_k^{(n)}(\zeta)-f^{(n)}(\zeta)}{n!}(z-\zeta)^n \Big|\\
  & \leq \sum_{n=0}^{N}\frac{|f_k^{(n)}(\zeta)-f^{(n)}(\zeta)|}{n!}|z-\zeta|^n \\
  & \leq \sum_{n=0}^{N} \frac{\|f_k^{(n)}-f^{(n)}\|_{K_1}}{n!}M^n
\end{align*}
 which means that
 \begin{equation*}
   \|S_N(f_k)-S_N(f)\|_{K} \leq \sum_{n=0}^{N} \frac{\|f_k^{(n)}-f^{(n)}\|_{K_1}}{n!}M^n
 \end{equation*}
  and therefore $S_N(f_k) \longrightarrow S_N(f)$ uniformly on $K$, for each $N=0,1,2,...$

Now fix $\zeta_0 \in \Omega$ and, for $N \geq 0$, define
\begin{align*}
  T_N: \holo & \rightarrow \hol(\mathbb{C}) \\
  f & \mapsto S_N(f,\zeta_0)(\cdot)
\end{align*}
Then each $T_N$ is linear and continuous in $\holo$ and
\begin{equation*}
  \mathcal{U}_{T_N} = \big\{ f \in \holo: S_N(f, \zeta_0) \text{ is unbounded in } \mathbb{C}\big\}
\end{equation*}
But $S_N(f, \zeta_0)$ is a polynomial, so it is bounded in $\mathbb{C}$ if and only if it is constant in $\mathbb{C}$. Therefore
\begin{equation*}
  \mathcal{U}_{T_N} = \big\{ f \in \holo: S_N(f, \zeta_0) \text{ is non-constant in } \mathbb{C}\big\}
\end{equation*}
For $N=0$ we have that $S_N(f,\zeta_0)(z)= f(\zeta_0)$, $z \in \mathbb{C}$, so $\mathcal{U_{T_N}}= \emptyset$. \\
for $N \geq 1$, we have that
\begin{equation*}
  S_N(f,\zeta_0)(z)= \sum_{n=0}^{N}\frac{f^{(n)}(\zeta_0)}{n!}(z- \zeta_0)^n
\end{equation*}

is constant if and only if $f^{'}(\zeta_0)= f^{''}(\zeta_0)=...=f^{(N)}(\zeta_0)=0$. But there always exists a function $f \in \holo$ such that $f^{(k)}(\zeta_0) \neq 0$, for all $k \in \mathbb{N}$, for example $f(z)= e^z$. Therefore, $\mathcal{U}_{T_N} \neq \emptyset$, for all $N \geq 1$. By Proposition (\ref{capuT_n}) we have that the set $\bigcap\limits_{N=1}^{\infty}\mathcal{U}_{T_N}$ of all the functions $f \in \holo$ with the property that the function $S_N(f, \zeta_0)$ is unbounded in $\mathbb{C}$ for all $N \geq 1$, is a dense $G_\d$ set in $\holo$. \\
\par
We mention that $\mathcal{U}_{T_1}$ is an open dense set in $\hol(\Omega)$ because $\mathcal{U}_{T_1}= \{f \in \holo : f^{'}(\zeta_0) \neq 0\}$. Similarly, $\mathcal{U}_{T_N} $ is also an open dense set in $\holo$, so $\bigcap\limits_{N=1}^{\infty} \mathcal{U}_{T_N}$ is $G_\d$ dense in $\holo$. So this corollary of Proposition (\ref{capuT_n}) is well known and obvious. A similar result holds if we replace $\mathbb{C}$ by any unbounded domain $\Omega_2$; in particular this holds for $\Omega_2 = \Omega$ if $\Omega$ is unbounded.
\par
Now fix $z=0$ and, for $N \geq 0$, define
\begin{align*}
  \ttn :\holo & \rightarrow \holo \\
  f & \mapsto S_N(f, \cdot)(0)
\end{align*}
Each $\ttn$ is linear and continuous in $\holo$. \par
For $N=0$, we have that $S_0(f,\zeta)(0)=f(\zeta)$, $\zeta \in \Omega$, and therefore
\begin{equation*}
  \mathcal{U}_{\ttn}= \big\{ f\in \holo: f \text{ is unbounded in } \Omega \big\}
\end{equation*}
which is a dense $G_\d$ set in $\holo $ by Proposition (\ref{a_0}). \par
For $N \geq 1$, if $\Omega = \mathbb{C}$, take $f(z)= e^z$, $z \in \mathbb{C}$. Since $z \mapsto e^z$ dominates the polynomials in $\mathbb{C}$, we have that $S_N(f,\zeta)(0)$ is unbounded in $\mathbb{C}$. If $\Omega \neq \mathbb{C}$, take $\zeta_0 \in \partial \Omega$ and $f(z)= \tfrac{1}{z-\zeta_0}$, $z \in \Omega$. Then $f \in \holo$ and
\begin{equation*}
  S_N(f,\zeta)(0)= \sum_{n=0}^{N} \frac{\zeta ^n}{(\zeta-\zeta_0)^{n+1}}, \hspace{5pt} \zeta \in \Omega
\end{equation*}
which is a rational function with poles only at $z= \zeta_0$. Hence $\lim \limits_{\zeta\rightarrow\zeta_0}|S_N(f,\zeta)(0)| = \infty$ and $S_N(f,\cdot)(0)$ is unbounded in $\Omega$.\\
Therefore, $\mathcal{U}_{\ttn} \neq \emptyset$ for all $N \geq 0$, so by Corollary (\ref{capuT_n}) we have that the set $\bigcap\limits_{N=0}^{\infty}\mathcal{U}_{\ttn}$ of all functions $f \in \holo$ with the property that $S_N(f,\cdot)(0)$ is unbounded in $\Omega$ for all $N \geq 0$, is a dense $G_\d$ set in $\holo$.
\par
Next we consider the following class $\so$ of functions on $\Omega$:
\begin{definition}
  Let $\Omega$ be an open, non-empty subset of $\mathbb{C}$. We define $\so$ to be the set of all functions $f \in \holo$ such that $\big\{  \ttn(f)\big\}_{N \geq 0}$ is dense in $\holo$.
\end{definition}
\label{so def}
From now on and unless otherwise stated we assume that $\Omega$ is a simply connected domain in $\mathbb{C}$.
Our goal is to show that either $\so = \emptyset$ or $\so$ is a dense $G_\d$ set in $\holo$. To this end, first observe that,  $\holo$ is separable: the set $\{p_j\}_j$ of all polynomials with coefficients having rational coordinates is dense in $\holo$ by the Runge Theorem. Now consider an exhaustive sequence $\{K_m\}_m$ of compact subsets of $\Omega$, i.e. a sequence $\{K_m\}_m$ of compact subsets of $\Omega$ such that
\begin{enumerate}
  \item $\Omega = \bigcup\limits_{m=1}^{\infty}K_m $
  \item $K_m$ lies in the interior of $K_{m+1}$, for $m=1,2,...$
  \item Every compact subset of $\Omega$ lies in some $K_m$
  \item Every component of $K_m^{\mathsf{c}}$ contains a component of $\Omega^{\mathsf{c}}$, $m=1,2,...$
\end{enumerate}
(See \cite{Rudin:1987:RCA:26851})
\par
Now we can show that $\so$ can be expressed as a set which will be shown to be a $G_\d$ one in $\holo$:
\begin{proposition}
\label{so=}
 $ \so = \bigcap\limits_{s,j,m=1}^{\infty}\bigcup\limits_{N=0}^{\infty}\big\{ f \in \holo: \sup\limits_{\zeta \in K_m}|\ttn (f)(\zeta)-p_j(\zeta)| < \frac{1}{s}\big\}$
\end{proposition}
\begin{proof}
  That $\so$ is a subset of the set on the right is an immediate consequence of the definition of $\so$.
   \par
   Consider now a function $f$ in the set on the right, a function $g \in \holo$, a compact subset $K$ of $\Omega$ and an $\epsilon >0$. There exists an $m \geq 1$ such that $K \subset K_m$ and an $s \geq 1$ such that $\tfrac{1}{s}< \epsilon$. For these $g$, $K_m$ and $s$, there exists a $j \geq 1$ such that
  \begin{equation*}
   \sup\limits_{\zeta \in K} |p_j(\zeta)-g(\zeta)| \leq \sup\limits_{\zeta \in K_m} |p_j(\zeta)-g(\zeta)|< \tfrac{1}{2s}
  \end{equation*}
For these $K_m$, $s$ and $j$, there exists an $N \geq 0$ such that
  \begin{equation*}
  \sup\limits_{\zeta \in K} |\ttn(f)(\zeta)- p_j(\zeta)| \leq  \sup\limits_{\zeta \in K_m} |\ttn(f)(\zeta)- p_j(\zeta)|< \tfrac{1}{2s}
  \end{equation*}
  By the triangle inequality, for $z \in K$, we have
  \begin{align*}
    |\ttn(f)(z)- g(z)|& \leq |\ttn(f)(z)- p_j(z)|+|p_j(\zeta)-g(\zeta)| \\
    & \leq \sup\limits_{\zeta \in K} |\ttn(f)(\zeta)- p_j(\zeta)| +\sup\limits_{\zeta \in K} |p_j(\zeta)-g(\zeta)|\\
    &< \frac{1}{2s} + \frac{1}{2s}
  \end{align*}
  Therefore, $\sup\limits_{\zeta \in K} |\ttn(f)(\zeta)- g(\zeta)| \leq \tfrac{1}{s}< \epsilon$, so $\{\ttn(f)\}$ is dense in $\holo$.
\end{proof}
\begin{proposition}
  \label{so g_d}
  $\so$ is a $G_\d$ set in $\holo$.
\end{proposition}
\begin{proof}
  By Proposition (\ref{so=}), it suffices to show that, for $j,s,m \geq 1$ and $N \geq 0$, the set
  \begin{equation*}
   E_{j,s,m,N}:= \big\{f \in \holo: \sup\limits_{\zeta \in K_m}|\ttn (f)(\zeta)-p_j(\zeta)| < \frac{1}{s} \big\}
  \end{equation*}
  is open in $\holo$.
   \par
   To this end, consider functions $g_k \in \holo$, $k \geq 1$, and $g \in E_{j,s,m,N}$ such that $g_k \longrightarrow g$ uniformly on compact subsets of $\Omega$. It suffices to find a $k_0$ such that $g_k \in E_{j,s,m,N}$, for all $k \geq k_0$. Since $g \in E_{j,s,m,N}$, there exists a $\d >0$ such that
  \begin{equation*}
    \sup\limits_{\zeta \in K_m} |\ttn (g)(\zeta)-p_j(\zeta)|< \tfrac{1}{s}- 2\d
  \end{equation*}
  Set $M = \max {\{e^{|\zeta|}:\zeta \in K_m}\}$. By the Weierstrass Theorem we have that $g_k^{(i)} \longrightarrow g^{(i)}$ uniformly on compact subsets of $\Omega$, $i=0,1,...,N$, so there exists a $k_0 \in \mathbb{N}$ such that
  \begin{equation*}
    \|g_k^{(i)}- g^{(i)}\|_{K_m}< \frac{\d}{M}
  \end{equation*}
  for all $i=0,1,...N$. Therefore, for $z \in K_m$ and $k \geq k_0$ we have
  \begin{align*}
    |\ttn (g_k)(z)-p_j(z)|& \leq |\ttn (g_k)(z)-\ttn(g)(z)|+|\ttn (g)(z)-p_j(z)| \\
    &=\Big| \sum\limits_{n=0}^{N}\frac{g_k^{(n)}(z)-g^{(n)}(z)}{n!}(-z^n) \Big|+|\ttn (g)(z)-p_j(z)|\\
    &\leq \sum\limits_{n=0}^{N}\frac{|g_k^{(n)}(z)-g^{(n)}(z)|}{n!}|z|^n +\sup\limits_{\zeta \in K_m} |\ttn (g)(\zeta)-p_j(\zeta)| \\
    &< \sum_{n=0}^{N}\frac{\|g_k^{(n)}-g^{(n)}\|_{K_m}}{n!}|z|^n +\frac{1}{s} -2\d \\
    &<\frac{\d}{M} \sum_{n=0}^{N}\frac{|z|^n}{n!}  +\frac{1}{s} -2\d \\
    & \leq  \frac{\d}{M} \sum_{n=0}^{\infty}\frac{|z|^n}{n!}  +\frac{1}{s} -2\d \\
    &=\frac{\d}{M} e^{|z|}  +\frac{1}{s} -2\d \\
    & \leq \frac{\d}{M} \hspace{3pt} M  +\frac{1}{s} -2\d \\
    &=\frac{1}{s}-\d
  \end{align*}
  Since the $z \in K_m$ was arbitrary, we have that
  \begin{equation*}
    \sup\limits_{\zeta \in K_m} |\ttn (g_k)(\zeta)-p_j(\zeta)|\leq \frac{1}{s}- \d <\frac{1}{s}
  \end{equation*}
  for all $k \geq k_0$. Hence $g_k \in E_{j,s,m,N}$, $k \geq k_0$. This completes the proof.
\end{proof}
\begin{proposition}
  \label{so either or}
  Let $\Omega$ be a simply connected domain in $\mathbb{C}$.
  Either $\so = \emptyset$ or $\so$ is a dense $G_\d$ set in $\holo$.
\end{proposition}
\begin{proof}
  If $\so \neq \emptyset$, by Proposition (\ref{so g_d}) it suffices to show that $\so$ is dense in $\holo$.
   \par
  Let $f \in \so$. Observe that, if $p$ is a polynomial, then $f+p \in \so$. Indeed, $f+p \in \holo$ and, for all $N> \deg p$, we have that $\ttn(f+p)= \ttn(f)+q_p$, where
  \begin{equation*}
    q_p(\zeta)= \sum_{n=0}^{N}\frac{(-1)^n p^{(n)}(\zeta)}{n!}\zeta ^n, \hspace{10pt} \zeta \in \Omega
  \end{equation*}
  is again a polynomial. For a function $g \in \holo$, we have that $g-q_p \in \holo$, and therefore there exists a sequence $\{\lambda_n\}$ in $\mathbb{N}$ such that $\ttln (f) \longrightarrow g- q_p$ uniformly on compact subsets of $\Omega$. But then $\ttln(f+p)=\ttln(f)+q_p \longrightarrow g$ uniformly on compact subsets of $\Omega$, i.e. $\{\ttn(f+p)\}$ is dense in $\holo$ and $f+p \in \so$. \\
  Now the density of $\so$ in $\holo$ follows easily because by Runge's Theorem the polynomials are dense in $\holo$.
\end{proof}
At this point observe that, if $0 \in \Omega$, then $\so = \emptyset$. Indeed, for $f,g \in \holo$ such that $f(0) \neq g(0)$, we have that, for any $N \in \mathbb{N}$ and any compact subset $L$ of $\Omega$ such that $0 \in L$,
\begin{equation*}
  \sup\limits_{\zeta \in L}|\ttn(f)(\zeta)-g(\zeta)| \geq |\ttn(f)(0)-g(0)|=|f(0)-g(0)|>0
\end{equation*}
so there is no subsequence of $\{\ttn(f)\}$ that converges to $g$ uniformly on compact subsets of $\Omega$. \\

\begin{definition}
  Let $\Omega$ be open in $\mathbb{C}$. The set $\sto$ is the set of all $f \in \holo$ with the property that, for every $c \in \mathbb{C}$ there exists a sequence $\{\lambda_n\}$ in $\mathbb{N}$ such that, for every $L \subset \Omega$ compact,
  \begin{equation*}
    \sup\limits_{\zeta \in L }|\widetilde{T}_{\lambda_n}(f)(\zeta)-c| \longrightarrow 0, \hspace{7pt} n \rightarrow \infty
  \end{equation*}
\end{definition}

\begin{proposition}
  The set $\sto$ is a $G_\d$ set in $\holo$.
\end{proposition}

\begin{proof}
  Let $\{z_j\}_{j \in \mathbb{N}}$ be an enumeration of the points in the complex plane with rational coordinates. Following the proof of Propositions (\ref{so=}) and (\ref{so g_d}), we get that
  \begin{equation*}
    \sto = \bigcap\limits_{s,j,m=1}^{\infty}\bigcup\limits_{N=0}^{\infty}\big\{ f \in \holo: \sup\limits_{\zeta \in K_m}|\widetilde{T}_{N} (f)(\zeta)-z_j| < \frac{1}{s}\big\}
  \end{equation*} and that the set
  \begin{equation*}
    \big\{ f \in \holo: \sup\limits_{\zeta \in K_m}|\widetilde{T}_{N} (f)(\zeta)-z_j| < \frac{1}{s}\big\}
  \end{equation*}
  is open in $\holo$, $m,j,s \geq 1$, $N \geq 0$.
\end{proof}
Observe again that, if $0 \in \Omega$, then $\sto = \emptyset$. Indeed, for $f \in \holo$, $c \in \mathbb{C}$ with $f(0) \neq c$ and $L \subset \Omega$ compact, we have that
\begin{equation*}
  \sup\limits_{\zeta \in L}|\widetilde{T}_N(f)(\zeta)-c| \geq |\widetilde{T}_N(f)(0)-c|= |f(0)-c| >0
\end{equation*}
for all $N \in \mathbb{N}$. However, we can show that $\sto$ is dense in $\holo$ if $\Omega $ is a simply connected domain and $0 \notin \Omega$:
\begin{theorem}
  \label{sto dense gd} Let $\Omega $ be a simply connected domain with $0 \notin \Omega$. Then $\sto$ contains a dense $G_\d$ set in $\holo$.
\end{theorem}
\begin{proof}
Since $\Omega$ is a simply connected domain, the class $\uo$ is a dense $G_\d$ set in $\holo$.
  We will show that $\uo \subset \sto$.
  \par
   Let $f \in \uo$ and $c \in \mathbb{C}$. Take $K=\{0\}$, which is disjoint from $\Omega$ because $0 \notin \Omega$. Then $K$ is a compact set in $\mathbb{C}$, $K\cap \Omega = \emptyset$, $K^{c}$ is connected, and the function $h(z)=c$, $z \in K$, is continuous on $K$ and (trivially) analytic in the interior of $K$. By definition of the class $\uo$, there exists a sequence $\{\lambda_n\}$ in $\mathbb{N}$ such that, for every compact set $L \subset \Omega$,
  \begin{equation*}
    \sup\limits_{\zeta \in L}\sup\limits_{z \in K}|S_{\lambda_n}(f,\zeta)(z)-h(z)| \longrightarrow 0, \hspace{10pt} n \rightarrow \infty
  \end{equation*}
  or
  \begin{equation*}
    \sup\limits_{\zeta \in L}|S_{\lambda_n}(f,\zeta)(0)-c| \longrightarrow 0, \hspace{10pt} n \rightarrow \infty
  \end{equation*}
  But this is exactly
  \begin{equation*}
    \sup\limits_{\zeta \in L}|\widetilde{T}_{\lambda_n}(f)(\zeta)-c| \longrightarrow 0, \hspace{10pt} n \rightarrow \infty
  \end{equation*}
  Therefore, $f \in \sto$. This completes the proof.
\end{proof}
\section{A more general statement}

\par
During a seminar on these topics, T. Hatziafratis posed the following question: Let $E$ be a countable dense subset of $\mathbb{T}=\{z \in \mathbb{C}: |z|=1 \}$. Is it true that, for the generic function $f \in \hol(\mathbb{D})$, all the derivatives and anti-derivatives of $f$ are unbounded on every radius joining $0$ to a point of $E$?
\par
The answer to this question is affirmative. To see this, we examine a more general case:
\begin{proposition}
\label{general either or}
Let $\Omega \subset \mathbb{C}$ be an open set, $X$ a non-empty subset of $\Omega$.  \newline If $T: \holo \rightarrow \holo$ is a linear operator with the property that, for every $z \in \Omega$, the mapping $ \holo \ni f \mapsto T(f)(z) \in \mathbb{C}$ is continuous, and
\begin{equation*}
  S=S(T,\Omega,X)=\{f \in \holo: T(f) \text{ is unbounded on } X\},
\end{equation*}
 then either $S = \emptyset$ or $S$ is a dense $G_\d$ set in $\holo$.
\end{proposition}
\begin{proof}
  \par
To show that $S$ is a $G_\d$ set, for $m \geq 1$, define
\begin{equation*}
  S_m=\{f \in \holo: \exists z \in X \text{ such that } |T(f)(z)|>m \}
\end{equation*}
Then $S= \bigcap\limits_{m=1}^{\infty} S_m$. Since the mapping $f \mapsto T(f)(z)$ is continuous, the set $S_m$ is open in $\holo$, for each $m \geq 1$. Hence, $S$ is a $G_\d$ set in $\holo$.
\par
To show that $S$ is dense in $\holo$ if it is not empty, let $g \in S$, i.e. $g \in \holo$ and $T(g)$ is unbounded on $X$, and let $f \in \holo$. If $T(f)$ is unbounded on $X$, then $f \in S$ and $f$ is (trivially) the limit in $\holo$ of a sequence of functions in $S$. If $T(f)$ is bounded on $X$ by, say, $M_1$, then, for a fixed $n \geq 1$, the function $T(f+\frac{1}{n} \hspace{2pt}g)$ is unbounded on $X$. Indeed, suppose it is bounded on $X$ by a positive number $M_2$. Then, if $z \in X$, by the linearity of $T$ we would have
\begin{align*}
  |T(g)(z)| & = n \hspace{2pt} |T(\frac{1}{n} \hspace{2pt} g)(z)| \\
            & = n \hspace{2pt} |T(f+\frac{1}{n} \hspace{2pt} g)(z)- T(f)(z)|\\
            &\leq n \hspace{2pt} |T(f+\frac{1}{n} \hspace{2pt} g)(z)|+n \hspace{2pt}|T(f)(z)|\\
            &\leq n \hspace{2pt} M_2 +n \hspace{2pt} M_1
\end{align*}
But this means that $T(g)$ is bounded on $X$ by $n \hspace{2pt}(M_1+M_2)$, which is contradictory to the fact that $T(g)$ is unbounded on $X$. Therefore, $T(f+\frac{1}{n} \hspace{2pt}g)$ is unbounded on $X$ for every $n \geq 1$; in other words $f+\frac{1}{n} \hspace{2pt}g \in S$, for every $n \geq 1$ . But $f+\frac{1}{n} \hspace{2pt}g \longrightarrow f$, $n \rightarrow \infty$, uniformly on compact subsets of $\Omega$, so $f$ is again the limit in $\holo$ of a sequence of functions in $S$. Since $f$ was an arbitrary function in $\holo$, $S$ is dense in $\holo$ and the proof is complete.
\end{proof}
\par
Consider now countable $T^{(k)}$ and $X_m$ such that $S(T^{(k)}, \Omega, X_m) \neq \emptyset$, for all $k,m$. Then Baire's Theorem gives that $\bigcap\limits_{k,m}S(T^{(k)}, \Omega, X_m)$ is a dense $G_\d$ set in $\holo$. This answers the aforementioned question in the affirmative, because if $\zeta_m \in E$ and $X_m$ is the radius joining $0$ to $\zeta_m$, then the function $g(z)= \frac{1}{z-\zeta_m}$, $z \in \mathbb{D}$, belongs to $S(T^{(k)},\mathbb{D},X_m)$ for all $k \geq 0$, where $T$ is the differentiation operator.
\par
More generally, we can replace $\mathbb{D}$ with any open non-empty set $\Omega$ in $\mathbb{C}$, $T$ being the differentiation operator and $X_m \subset \Omega$ having at least one accumulation point in $\partial \Omega$. If $\Omega $ is simply connected, then we obtain the analogous result for both the integration operator and the operator related to Taylor partial sums $\ttn$ that was defined before.
\par
Observing that in the proof of Proposition (\ref{general either or}) no properties of $\holo$ were used other than those of a topological vector space, we can obtain the best generalization of our result, where completeness is not assumed and the proof does not use Baire's Theorem:
\begin{proposition}
  \label{either or topological vs}
  Let $\mathcal{V}$ be a topological vector space over the field $\mathbb{R}$ or $\mathbb{C}$ and $X$ a non-empty set. Denote by $F(X)$ the set of all complex-valued functions on $X$ and consider a linear operator $T: \mathcal{V} \rightarrow F(X)$ with the property that, for all $x \in X$, the mapping $ \mathcal{V} \ni \alpha \mapsto T(\alpha)(x) \in \mathbb{C}$ is continuous. Let $S=\{\alpha \in \mathcal{V}: T(\alpha) \text{ is unbounded on } X\}$. Then either $S= \emptyset$ or $S$ is a dense $G_\d$ set in $\mathcal{V}$.
\end{proposition}
 \begin{proof}
   That $S$ is a $G_\d$ set follows from the fact that $S = \bigcap\limits_{m=1}^{\infty} \bigcup\limits_{x \in X}\{\alpha \in \mathcal{V}: |T(\alpha)(x)|> m\}$ and the continuity of $\alpha \mapsto T(\alpha)(x)$. The proof that $S$ is dense if it is non-empty is identical to the proof of Proposition (\ref{general either or}).
 \end{proof}

\vspace{3pt}
\textit{Acknowledgement}---
The topics discussed in this article were suggested by V. Nestoridis. I would like to thank him for the guidance and the insightful suggestions offered. I would also like to thank T. Hatziafratis for his taking interest in the topics discussed.

\bibliography{Arxiv11}
\bibliographystyle{plain}
\begin{flushleft}
  Department of Mathematics \\
National and Kapodistrian University of Athens\\
Panepistimiopolis, 157-84\\
Athens\\
 Greece\\
\vspace{5pt}
e-mail: siskakmaria@math.uoa.gr
\end{flushleft}

\end{document}